\documentclass[a4paper,10pt]{amsart}

\usepackage{inputenc}
\usepackage{color}
\usepackage{graphicx}
\usepackage{times,mathptm}
\usepackage{amsfonts,amscd,amssymb,amsmath}
\usepackage{xypic}
\usepackage{latexsym}
\usepackage{mathdots}
\usepackage{tikz}
\usetikzlibrary{arrows,snakes,backgrounds}

\newtheorem{proposition}{Proposition}[section]
\newtheorem{lemma}[proposition]{Lemma}
\newtheorem{corollary}[proposition]{Corollary}
\newtheorem{theorem}[proposition]{Theorem}

\theoremstyle{definition}
\newtheorem{definition}[proposition]{Definition}
\newtheorem{example}[proposition]{Example}

\theoremstyle{remark}
\newtheorem{remark}[proposition]{Remark}

\newcommand{\proplabel}[1]{\label{prop:#1}}
\newcommand{\propref}[1]{Proposition~\ref{prop:#1}}

\newcommand{\lemlabel}[1]{\label{lem:#1}}
\newcommand{\lemref}[1]{Lemma~\ref{lem:#1}}

\newcommand{\thelabel}[1]{\label{the:#1}}
\newcommand{\theref}[1]{Theorem~\ref{the:#1}}

\newcommand{\corlabel}[1]{\label{cor:#1}}
\newcommand{\corref}[1]{Corollary~\ref{cor:#1}}

\newcommand{\deflabel}[1]{\label{def:#1}}
\newcommand{\defref}[1]{Definition~\ref{def:#1}}

\newcommand{\exalabel}[1]{\label{ex:#1}}
\newcommand{\exaref}[1]{Example~\ref{ex:#1}}

\def\Ker{{\rm Ker}}

\def\Mod{{\rm Mod}}

\def\ot{\otimes}

\def\-{{\rm-}}

\def\dim{{\rm dim}}

\def\ov{\overline}
\def\el{\rm el}

\def\ann{\rm ann}
\def\Irr{\rm Irr}
\def\Aut{\rm Aut}
\def\Ind{\rm Ind}

\def\mapright#1{\smash{\mathop{\longrightarrow}\limits^{#1}}}

\def\hookmapright#1{\smash{\mathop{\hookrightarrow}\limits^{#1}}}

\newcommand{\dsubset}{\mathrel{
\rotatebox[origin=c]{45}{$\subset$}}
}

\newcommand{\dsubsetneq}{\mathrel{
\rotatebox[origin=c]{45}{$\subsetneq$}}
}

\makeindex

\begin{document}
\title{Glider representations of group algebra filtrations of nilpotent groups}
\author[F. Caenepeel]{Frederik Caenepeel}
\address{Department of Mathematics, University of Antwerp, Antwerp, Belgium}
\email{Frederik.Caenepeel@uantwerpen.be}
\author[F. Van Oystaeyen]{Fred Van Oystaeyen}
\address{Department of Mathematics, University of Antwerp, Antwerp, Belgium}
\email{Fred.Vanoystaeyen@uantwerpen.be}

\begin{abstract}
We continue the study of glider representations of finite groups $G$ with given structure chain of subgroups $e \subset G_1 \subset \ldots \subset G_d = G$. We give a characterization of irreducible gliders of essential length $e \leq d$ which in the case of $p$-groups allows to prove some results about classical representation theory. The paper also contains an introduction to generalized character theory for glider representations and an extension of the decomposition groups in the Clifford theory. Furthermore, we study irreducible glider representations for products of groups and nilpotent groups.
\end{abstract}

\thanks{The first author is Aspirant PhD Fellow of FWO}
\maketitle

\section{Introduction}

In \cite{CVo1} we introduced glider representations of a finite group and started the study of a Clifford theory for these. The category of finitely generated glider representations of a group $G$ with structure chain $e \subset G_1 \subset \ldots \subset G_d = G$ is semisimple, i.e. each glider may be written as the fragment direct sum of irreducible gliders, see \cite{EVO}. The chain in $G$ defines a finite algebra filtration of length $d$ on $KG$ by $F_{-n}KG = 0$ for $n > 0, F_iKG = KG_i$ for $i$ between 0 and $d$, $F_mKG = KG$ for $m \geq d$. For such filtrations the irreducible gliders of length $e \leq d$ are easily studied.\\
The paper starts with a short preliminary section about irreducible gliders. In Section 3 we first deduce some additional theory concerning irreducible gliders, mainly relating to $V \oplus \ldots \oplus V = V^n$ for some irreducible glider $V$. We obtain\\

{\bf \theref{1}}
\emph{Let $V \in \Irr(G)$ and let $v_1,\ldots, v_n \in V$. If $a = v_1 + \ldots + v_n \in V^n$, then $KGa \cong V^m$ for some $m \leq n$ if and only if $\dim_K <v_1,\ldots, v_n> = m$.}\\

This allows treatment of gliders of type $M \supset Ka$. When considering chains of longer lengths 
the following problem appears: let $H \subset G$ be a subgroup, $U$ a simple representation of $G$ and $V$ a simple $H$-representation appearing in the decomposition of $U$ as an $H$-representation. Is it possible to have $V$ appear in different $U$ and $U'$ such that $U$ and $U'$ have different dimension? We first consider $p$-groups, then we obtain\\

{\bf \theref{pgroup}}
\emph{Let $G$ be a $p$-group and let $e \subset Z(G) \subset G_2 \subset \ldots \subset G_d = G$ be a maximal chain of normal subgroups. If $~V, W \in \Irr(G)$ are such that they lie over the same irreducible representation of $G_{d-1}$, then $\dim(V) = \dim(W)$.}\\

As an immediate corollary, we get the same result for finite nilpotent groups. This also shows that the Hasse diagram is not connected. The second part of Section 3 contains a definition of the character or generalized trace map for a glider representation. The example of $Q_8$ and $D_8$, the notorious example of a pair of groups establishing some shortcoming of classical character tables, is now given to show that gliders and the generalized character do discern between these groups! In \defref{gentrace} we define the generalized trace map of a glider $M$ of essential length $d$ as a map $\chi_M: ~G \to M_{d+1}(K)$ mapping $g$ to a lower triangular matrix of suitable module characters associated to some $KG_iM_j$ as $KG_i$-modules. We obtain\\

{\bf \propref{anti-diagonal}}
\emph{Let $\Omega \supset M$ be a glider representation with respect to a finite group algebra filtration. Then $M$ is irreducible of essential length equal to the filtration length and with $B(M) = 0$ if and only if the matrix $\chi(e)$ is symmetric with respect to the anti-diagonal and has a 1 in the left most corner.}\\

This only forms the beginning of a generalized character theory for gliders and is work in progress.\\

Now in Section 4 we reconsider decomposition groups but here we look at decomposition groups associated to each square in the ladder relating the two chains, so each square leads to five decomposition groups:
$$\begin{array}{ccccc}
H_{i+1} & \subset & G_{i+1}' & \subset & G_{i+1}\\
\cup &&  & \dsubset~~~& \cup\\
H_i' & & H_i^2 && G_i''\\
\cup &~~~ \dsubset &&& \cup\\
H_i & \subset &G_i' & \subset & G_i 
\end{array}$$ 

We study the relations between these in \propref{4.1}, \propref{4.2}, \propref{4.3} and we arrive at\\

{\bf \theref{4.4}}
Let $G_2$ be a $p$-group and $S$ some building block such that its associated decomposition graph takes the form \eqref{situatie1} or \eqref{situatie2}. Then $G_2 \cong H_1 \rtimes (C_p \times C_p)$ where the semi-direct product is not a direct product. Moreover $F = H_1 \rtimes <\ov{g}> = H_1 \rtimes <\ov{z}>$ for some $g, z \in G_2 - (H_2 \cup G_1)$ such that $F \cong H_1 \rtimes (<\ov{g},\ov{z}>) \subset H_1 \rtimes G_2/H_2 \times G_2/G_1$.\\

Finally, Section 5 deals with nilpotent groups of order $p^kq^l$. The section contains characterizations for irreducible gliders to be isomorphic to a tensor product of gliders in terms of triples defined just before \propref{a=bc}. The main result we obtain is\\

{\bf \theref{char}}
Let $GH$ be a nilpotent group of order $p^\alpha q^\beta$. TFAE
\begin{enumerate}
\item $GH$ is abelian;
\item Every irreducible $FKGH$ glider $M$ is isomorphic to the tensor product $M_1 \ot M_2$ of an $FKG$- and an $FKH$-glider if and only if the associated triple $(a,b,c)$ of $M$ satisfies $a = bc$.
\end{enumerate}

 We started the theory of gliders for groups in order to obtain an extension of representation theory, that is, we look for new structural results involving classical representation theory but fitting in our new structure theory. This paper is a further step in this direction, opening new gateways for further research too. 

\section{Preliminaries}

Irreducible fragments appeared for the first time in \cite{EVO}. Since the authors slightly altered the definition of both fragments and glider representations (see \cite{CVo1},\cite{CVo}), they reintroduced the notion of irreducibility in \cite{CVo}. For convenience of the reader, we quickly recall some facts.\\

 A subfragment $N$ of $M$ is said to be trivial if either:\\
$T_1$. There is some $i \in \mathbb{N}$ such that $N_i = B(N)$ but $M_i \neq B(M)$.\\
$T_2$. There is a $i \in \mathbb{N}$ such that $N_i = 0$ but $M_i \neq 0$.\\
$T_3$. There exists a monotone increasing map $\alpha: \mathbb{N} \to \mathbb{N}$ such that $N_i = M_{\alpha(i)}$ and $\alpha(i) - j \geq \alpha(i - j)$ for all $j \leq i$ in $\mathbb{N}$.\\

A fragment (or glider representation) $M$ is called irreducible exactly when all subfragments are trivial. In the proof of \propref{anti-diagonal} below we need a result on irreducible fragments that already appeared in \cite{EVO}, but needs to be reproved with our improved definitions:  The group algebra filtrations $FKG$ we will be dealing with are right bounded filtrations. A filtration $FR$ is called right bounded if for some $n \in \mathbb{N},~ F_nR = F_{n+1}R = \ldots = R$. If $M \supset M_1 \supset \ldots \supset M_n \supset \ldots$ is an $FR$-fragment then for $N \geq n,~ F_nRM_N = RM_N \subset M_{N-n} \subset M$. For any $m,~M_m \supset M_{m+1} \supset \ldots$ is a subfragment of $M$ and if $L$ is a subfragment of $M_m$:
$$\begin{array}{ccccccc}
L & \supset & L_1 & \supset & L_2 & \supset & \ldots \\
\cap && \cap && \cap &&\\
M & \supset & M_1 & \supset & M_2 & \supset & \ldots
\end{array}$$
then $M \supset \ldots \supset M_{m-1} \supset L \supset L_1 \supset L_2 \supset \ldots$ is a subfragment of $M$. The result we need is

\begin{lemma}\lemlabel{irr1}
Let $FR$ be a right bounded filtered ring. If $M \supset M_1 \supset \ldots \supset M_n \supset \ldots$ is an irreducible $FR$-fragment then for any $m$, the fragment $M_m \supset M_{m+1} \supset \ldots$ is also irreducible.
\end{lemma}
\begin{proof}
If $L$ is a subfragment of $M_m$ then $M(L) : ~ M \supset \ldots \supset M_{m-1} \supset L \supset L_1 \supset \ldots$ is trivial in $M$ by irreducibilty. Suppose that $B(M(L)) = B(L) = M(L)_n$ but $M_n \neq B(M) = B(M_m)$ for some $n$, then if  $n \geq m$ we have $L_{n-m} = B(L)$ and $(M_m)_{n-m} = M_n \neq B(M_m)$, so $L$ is trivial in this case. If $ n < m$, then $B(L) = M_n \neq B(M)$, whence $B(L) = B(M)$ and $M_n = B(M)$. In this case, $M_m$ is the trivial chain $B(M) \supset B(M) \supset \ldots$ and $L = M_m$, so $L$ is again trivial. If $M(L)$ is trivial of type $T_2$, then $M(L)_n = 0,~ M_n \neq 0$, for some $n$. If $n \geq m$, then $L_{n-m} = 0$, but $(M_m)_{n-m} = M_n \neq 0$. Hence $L$ is trivial. If $ n < m$, then $M_m = 0$ and $L$ is trivial. Finally, if $M(L)_n = M_{\alpha(n)}$ for all $n$ for some $\alpha: \mathbb{N} \to \mathbb{N}$ satisfying the conditions of $T_3$, then $M(L)_{m+k} = M_{\alpha(m+k)} = L_{\alpha(m+k) - m}$ for all $k$. The monotonic ascending function $\alpha_m(k) = \alpha(m+k) - m$ satisfies the condition of $T_3$ and shows that $L$ is trivial in $M_m$. This covers all the cases.
\end{proof}

\section{Irreducible gliders}

Throughout we consider a finite group $G$ together with some (normalizing) chain of subgroups $e \subset G_1 \subset \ldots \subset G_d = G$, which yields a group algebra filtration $FKG$ on $KG$, $K$ some field, by setting $F_iKG = KG_i$ for $0 \leq i \leq d, F_{-n}KG = 0$ for $n > 0$ and $F_{n}KG = KG$ for $n \geq d$. Furthermore, we assume that $K = \ov{K}$ of characteristic zero.
In \cite{EVO} it is shown that the category of finitely generated glider representations for some chain $e \subset G_1 \subset \cdots \subset G_d =G$ of groups is semisimple, that is, each glider representation can be written as the fragment direct sum of irreducible gliders. For these particular algebra filtrations of length $d$ the irreducible gliders of length $e \leq d$ are of a very specific form. Indeed, we recall from \cite{CVo}
\begin{lemma}\lemlabel{el}
Let $FA$ be a finite algebra filtration on a $K$-algebra $A$ and let $M$ be a weakly irreducible $FA$-fragment such that $M \neq B(M)$, then there is an $e \in \mathbb{N}$ such that $M_e \neq B(M)$ and $e$ is maximal as such. For this $e$, we have that $M_i = F_{e-i}AM_e$, for $0 \leq i \leq e$.
\end{lemma}

Even in the case of a chain of length 1, i.e. $ e \subset G$, the question arises what the irreducible gliders of essential length 1 are. In this simple situation, such an irreducible glider takes the form
$$ \Omega = M \supset M_1 = Ka$$
such that $KGa = M$. By Maschke's theorem, the group algebra $KG$ is semisimple and decomposes into 
\begin{equation}\label{KG}
KG = M_{n_1}(K) \oplus \cdots \oplus M_{n_k}(K),
\end{equation}
where $k$ is the number of conjugacy classes of $G$. Moreover, there are (up to isomorphism) $k$ irreducible representations $V_i$ of dimension $n_i$ and in particular the order of the group $G$ equals $\sum_{i=1}^k n_i^2.$\\

Now suppose that $M \supset Ka$ is an irreducible glider, then $M = \oplus_{i=1}^k V_i^{m_i}$ and $a \in M$ is such that it generates $M$. To answer the question for which $M$ and $a$ we obtain irreducible gliders, we recall the notion of (left) annihilator ideals. Let $a \in M$, then 
$$\ann(a) = \{ x \in KG~|~ xa = 0\}.$$
This is a left ideal of the group algebra. The left ideal of annihilators of the whole of $M$,
$$\ann(M) = \{ x \in KG~|~xM = 0\}$$ 
is in fact a two-sided ideal and we have that 
$$ \ann(M) = \bigcap_{m \in M} \ann(m).$$
If $M = V_i$ is an irreducible module, then it follows immediately from \eqref{KG} that 
$$\ann(V_i) = M_{n_1}(K) \oplus \cdots M_{n_{i-1}}(K) \oplus  M_{n_{i+1}}(K)\oplus \cdots \oplus M_{n_k}(K),$$
whence $\dim_K(\ann(V_i)) = |G| - n_i^2$.
If $U$ is a $K$-vector space, we denote by $\mathbb{P}U$ the associated projective space. The following lemma will be crucial
\begin{lemma}\lemlabel{Schur}
Let $U \in \Irr(G), u, u' \in U$. Then $\ann(u) = \ann(u') \Leftrightarrow \ov{u} = \ov{u'} \in \mathbb{P}U$.
\end{lemma}
\begin{proof}
The `only if' direction is clear, since $K \subset Z(KG)$. Suppose that the left annihilators are equal. Since $U$ is simple, we have that $U = KGu = KGu'$. Define a morphism 
$$\varphi: KGu \to KGu',~ xu \mapsto xu'.$$
This map is well-defined exactly because the annihilators are equal. Schur's lemma then yields the result.
\end{proof}

We start by looking at gliders $M \supset Ka$ for which $M = V^m$, with $V$ an $n$-dimensional simple representation. To begin, consider $V \oplus V$. For every $a = v_1 + v_2 \in V \oplus V$, we obtain a submodule $KGa \subset V \oplus V$. By the uniqueness of decomposition, this $KGa$ is isomorphic to $V$ or $V \oplus V$. 
\begin{proposition}\proplabel{twee}
Let $V \in \Irr(G)$ and suppose that $a = v_1 + v_2 + \ldots + v_m \in V^m$ is such that $KGa \cong V \subset V^m$, then all $\ov{v_i}$ are equal in $\mathbb{P}V$.
\end{proposition}
\begin{proof}
We consider first the case $m=2$, so suppose that $\ann(v_1) \neq \ann(v_2)$ and take some $x \in \ann(v_1)$, which does not annihilate $v_2$. Then $ 0 \neq xa = xv_2 \subset 0 \oplus V$. But since $V = KGxv_2$, we obtain that $0 \oplus V \subset KGa$, a contradiction. So both annihilators are equal and the result follows from \lemref{Schur}. Now consider the general case $m>2$. If $v_1 = 0$, then $KGa \subset V^{m-1}$, so we may assume that all $v_i \neq 0$. For any $i \neq j$, consider the projection of $V^m$ onto the $i$-th and $j$-th factor. Then $KGa$ is projected to $KG(v_i + v_j)$. Since $v_i + v_j$ is not zero in $V \oplus V$, we have that $KG(v_i+v_j) \cong V$ and the case $m=2$ then shows that $\ov{v_i} = \ov{v_j}$. The result follows.
\end{proof}
\begin{proposition}\proplabel{basis}
Let $V \in \Irr(G)$ be $n$-dimensional and let $\{v_1, \ldots, v_n\}$ be a $K$-basis for $V$. Then for $a = v_1 + \cdots + v_n \in V^{n}$ we have that $KGa = V^{n}$.
\end{proposition}
\begin{proof}
We have that $\ann(V) = \cap_{i=1}^n\ann(v_i).$ If $KGa \subset V^{n}$ is a proper submodule, it follows that $KGa \cong V^{m}$ for some $m < n$. Hence its $K$-dimension is $nm$. Consider the following short exact sequence of $K$-vectorspaces
$$ 0 \rightarrow \ann(a) \rightarrow KG \rightarrow KGa \rightarrow 0$$
We have that 
$$\ann(a) = \bigcap_{i=1}^n \ann(v_i) = \ann(V),$$
hence its $K$-dimension is $|G| - n^2$. It follows that the $K$-dimension of $KGa$ must be $n^2$, a contradiction.
\end{proof}
\begin{remark}
Let $\{v_1,\ldots,v_n\}$ be a $K$-basis for $V$. The short exact sequence
$$ 0 \to \ann(v_i) \to KG \to KGv \to 0$$
shows that $\dim(\ann(v_i)) = |G| -n$. This also follows from the decomposition of the group algebra $KG$ into a direct product of matrix algebras. Indeed, for $v_1$, we are searching for matrices in $M_n(K)$ that kill $(1,0,\ldots,0)^t$ and its dimension as a $K$-space is exactly $n^2 - n$. This reasoning allows to deduce that
$$\dim(\bigcap_{i=1}^m \ann(v_i)) = |G| - mn.$$
\end{remark}
As a corollary, we obtain a generalization of \propref{basis}
\begin{corollary}\corlabel{basis}
Let $V \in \Irr(G)$ be $n$-dimensional and let $\{v_1, \ldots, v_m\}$ be $m \leq n$ linear independent vectors in $V$ . Then for $a = v_1 + \cdots + v_m \in V^{m}$ we have that $KGa = V^{m}$.
\end{corollary}

We denote by $\pi_i$ the projection from $V^{n}$ onto the $i$-th factor and by $\pi_{\hat{i}}$ the projection onto $V^{n-1}$ which misses the $i$-th factor. Moreover, we denote by $V_i$ the $i$-th component in $V^n$, that is, $V_i = 0 \oplus \cdots \oplus 0 \oplus V \oplus 0 \oplus \cdots \oplus 0$, where the $V$ is in the $i$-th place.
\begin{lemma}
Let $V \in \Irr(G)$ and suppose that $a = v_1 + v_2 \in V \oplus V$. Then $KGa \cong \pi_i(KGa)$ if and only if $KGa \cong V$ for $i=1,2$.
\end{lemma}
\begin{proof}
If $KGa = V \oplus V$, then $\pi_i(KGa) = V$. If $KGa \cong V$, then the embedding is diagonal by \propref{twee}.
\end{proof}
\begin{lemma}\lemlabel{projection}
Let $V \in \Irr(G)$ and suppose that $a = v_1 + \cdots + v_n$. Then $v_i \in <v_j, j\neq i>$ if and only if $\pi_{\hat{i}}(KGa) \cong KGa$.
\end{lemma}
\begin{proof}
If $\pi_{\hat{i}}(KGa) \not\cong KGa$, then $\Ker(\pi_{\hat{i}}) \cap KGa = V_i \cap KGa \neq 0$. Hence, we find some $x \in KG$ such that $(0,\ldots,0,xv_i,0,\ldots,0) \in KGa$. If however, $v_i$ is a linear combination of the $v_j$, then any $x$ that annihilates all $v_j$, also kills $v_j$, hence $v_i \notin <v_j, j \neq i>$. Conversely, suppose that $v_i \notin <v_j, j \neq i>$. We have that $ 0 \neq \ann(V) \subseteq \bigcap_{j \neq i} \ann(v_j)$. If $\ann(v_i) = \bigcap_{j \neq i} \ann(v_j)$ then we can define a morphism 
$$\varphi : KGv_i \to KGv_1 \oplus \cdots \oplus \hat{KGv_i} \oplus KGv_{i+1} \oplus \cdots \oplus KGv_n,~ v_i \mapsto (v_1,\ldots, \hat{v_i}, \ldots, v_n).$$
By Schur's lemma we know that $\varphi = (\lambda_1,\ldots,\lambda_n)$ for some $\lambda_j \in K$. Hence, $v_i = \lambda_jv_j$, a contradiction. So $\ann(v_i) \neq \bigcap_{j \neq i} \ann(v_j)$ and we find an element $x \in KG$ that annihilates all $v_j$, except for $v_i$. The element $xa$ then sits in the kernel of $\pi_{\hat{i}} \cap KGa$.
\end{proof}
In particular, if $V$ is $n$-dimensional, and we work in $V^{n+1}$, then for any choice of $n+1$ nonzero vectors in $V$, we have that $\pi_{\hat{1}}(KGa) \cong KGa$. Since the projection is inside $V^n$, $KGa$ can at most be $V^n$. 
\begin{theorem}\thelabel{1}
Let $V \in \Irr(G)$ and let $v_1, \ldots, v_n \in V$. If $a = v_1 + \cdots + v_n \in V^n$, then $KGa \cong V^m$ for some $m\leq n$ if and only if $ \dim(<v_1,\ldots,v_n>)= m$.
\end{theorem}
\begin{proof}
Suppose that $V$ is $l$-dimensional. By the remarks above the theorem, we may reduce to $n \leq l$.  The cases $(1,n)$ for all $n \leq l$ follow from \propref{twee}, so we proceed by induction on $m$.  Suppose that the result holds for $m-1$ and suppose that $KGa \cong V^m \subset V^n$. Consider the projection $\pi_{\hat{1}}(KGa) = KG(v_2 + \cdots + v_n)$. If this is isomorphic to $V^{m-1}$, then by induction we know that $\dim(<v_2,\ldots,v_n>) = m-1$ and the claim follows by \lemref{projection}. If however, $\pi_{\hat{1}}(KGa) \cong KGa$, then $v_1 \in <v_2,\ldots, v_n>$. Then look at $\pi_{\hat{2}}(\pi_{\hat{1}}(KGa))$. If this is not isomorphic to $KGa$, then it is isomorphic to $V^{m-1}$ and by induction we have that $\dim(<v_1,\ldots,v_n>) = \dim(<v_2,\ldots,v_n>) = m$. If however, $\pi_{\hat{2}}(\pi_{\hat{1}}(KGa)) \cong KGa$, then we can remove $v_2$ as before. If $\pi_{\hat{n}} \circ \cdots \circ \pi_{\hat{1}}(KGa) \cong KGa$, then $KGa \cong V$ and \propref{twee} entails that $\dim(<v_1,\ldots,v_n>) = 1$.\\
Conversely, suppose that $\dim(<v_1,\ldots,v_n>)= m$. Up to reordering, we may assume that $\{v_1,\ldots, v_m\}$ are linearly independent. By \lemref{projection}, we obtain that $\pi_{\hat{n}} \circ \cdots \circ \pi_{\hat{m+1}}(KGa) \cong KGa$, and the result follows from \corref{basis}.
\end{proof}
\begin{corollary}
A glider of the form $V^m \supset Ka$ with $V \in \Irr(G)$ $n$-dimensional and $ m > n$ is never irreducible.
\end{corollary}

The general situation where $M = \oplus_i V_i^{m_i}$ does not cause any further difficulties. Indeed, we have 
\begin{proposition}
Let $V_1,\ldots,V_n \in \Irr(G)$ be non-isomorphic irreducible $G$-representations. If $a = v_1 + \ldots + v_n \in \bigoplus_{i=1}^n V_i $ is such that $KGa \cong \bigoplus_{j \in I}V_j \subset \bigoplus_{i=1}^n V_i $ for some proper subset $I \subset \{1,\ldots,n\}$, then $v_k=0$ for all $k \notin I$.
\end{proposition}
\begin{proof}
Let $\varphi$ be an isomorphism between $\bigoplus_{j \in I}V_j$ and $KGa$ and let $w = (w_j)_{j \in I} \in  \bigoplus_{j \in I}V_j$ be such that $\varphi(w) = a$. If $ v_k \neq 0$ for some $k \notin I$, we obtain a non-zero morphism 
$$ \bigoplus_{j \in I}V_j \mapright{\varphi} KGa \hookmapright{} \bigoplus_{i=1}^n V_i\mapright{\pi} V_k,$$
since $w \mapsto a \mapsto v_1 + \ldots + v_n \mapsto v_k$. This is a contradiction.
\end{proof}
\begin{corollary}
A glider of the form $V \oplus V' \supset Ka$ with $V,V' \in \Irr(G)$ non-isomorphic and $a = v + v', v \neq 0, v' \neq 0$ is irreducible.
\end{corollary}
All this leads to the following

\begin{theorem}\thelabel{irred}
Let $G$ be a finite group, $K$ an algebraically closed field of characteristic zero and let $\{V_1,\ldots,V_n\}$ be a full set of irreducible $G$-representations of resp. dimension $n_i$. A glider representation 
$$ M = \bigoplus_{i=1}^n V_i^{\oplus m_i} \supset Ka,$$
with $a = v_1^1 + \cdots + v_{m_1}^1 + v_1^2 + \cdots + v_{m_2}^2 + \cdots + v_1^n + \cdots + v_{m_n}^n \in M$ is irreducible if and only if 
\begin{enumerate}
\item $\forall i~ m_i \leq n_i$
\item $\forall i~ \dim(<v_1^i,\ldots,v_{m_i}^i>) = m_i$
\end{enumerate}
\end{theorem}

When studying glider representations we would like to work with chains of bigger length, so consider $e \subset G_1 \subset \ldots \subset G_d = G$. If $\Omega \supset M \supset \ldots \supset M_d \supset 0 \supset \ldots$ is an irreducible glider of $\el(M) = d$, then in particular, we must have that $M_d = Ka$ is a one-dimensional vector space, $M = KGa$ and \theref{irred} gives restrictions on the $G$-module $M$. The smaller groups $G_i$ appearing in the chain then determine the $M_{d-i}$. The following question now arises: If $\Omega \supset M$ is a bodyless irreducible glider of essential length $d$ with $\dim(M_1) = n_1$, what are the possible dimensions for $M$ and vice-versa?  For starters, we have

 \begin{lemma}\lemlabel{dim}
Let $e \subset G_1 \subset \ldots \subset G_d = G$ be a chain of groups and let $\Omega \supset M \supset \ldots \supset M_{d-1} \supset M_d = Ka$ be an irreducible glider representation. If $M_{d-i} = KG_{i}a = Ka$ for some $i$, then $\dim(M) \leq [G:G_i]$.
\end{lemma}
\begin{proof}
Write $G = G_i \sqcup g_2G_i \sqcup \ldots \sqcup g_mG_i$. Since $ha \in Ka$ for all $h \in G_i$, we have that $g_iha = k(h)g_i a$ and $g_ih'a = k(h')g_i a$ are linearly dependent for all $h,h' \in G_i$, hence $\dim(M) = \dim(KGa) \leq m = [G:G_i]$.
\end{proof}

\begin{remark}
In a similar vein one proves the more general statement that $\dim(M) \leq [G:G_i]\dim(M_{d-i})$.
\end{remark}

In fact, the answer follows from the relation between the simple $G_i$-representations. If $H < G$ is some subgroup, then we have a forgetful functor $U: \Mod(G) \to \Mod(H)$, which is essentially surjective. Let $U$ be a simple $G$-rep, if $V$ is a simple $H$-rep appearing in the decomposition into simples of $U$ as an $H$-rep, then we depict this by\\

\begin{tikzpicture}[scale=.5]
\node (U) at (0,2) {$U$};\\
\node (V) at (0,0) {$V$};\\
\draw (U) -- (V);
\end{tikzpicture}

and we say that $U$ lies over $V$. Of course, if we decompose $U \cong U_1 \oplus \ldots \oplus U_n$ into $H$-reps then we obtain\\

\begin{tikzpicture}[scale=.5]
\node (U) at (0,2) {$U$};
\node (U_1) at (-3,0) {$U_1$};
\node (U_2) at (-1,0) {$U_2$};
\node (dots) at (1,0) {$\ldots$};
\node (U_n) at (3,0) {$U_n$};
\draw (U) -- (U_1);
\draw (U) -- (U_2);
\draw (U) -- (dots);
\draw (U) -- (U_n);
\end{tikzpicture}

In this way we can draw a diagram relating all the simple representations of $G$ and $H$. We wonder whether we can encounter\\

\begin{tikzpicture}[scale=.5]
\node (U) at (-2,2) {$U$};
\node (V) at (2,2) {$U'$};
\node (W) at (0,0) {$V$};
\draw (U) -- (W) -- (V);
\end{tikzpicture}\\

where $U$ and $U'$ are not of equal dimension. We restrict to the situation of $p$-groups. It is well-known that they have a non-trivial center. Moreover, the index of the center in $G$ can not be $p$ (otherwise $G$ would be abelian already), so assume that $[G:Z(G)] = p^2$ and assume there is some chain $e \subset Z(G) \subsetneq G_2 \subsetneq G$. By the same reasoning as before, $G_2$ must be abelian. Let $V_1 \in \Irr(G_2)$ and suppose that there exist $W_1, P \in \Irr(G)$ of resp. dimension 1 and $p$ and such that $W_1 \cong V_1$ as $G_2$-reps and $P \cong V_1 \oplus gV_1 \oplus \ldots g^{p-1}V_1$ as $G_2$-reps for some $g \in G - G_2$. Again, since $G/G_2$ is cyclic of order $p$, there exists $p-1$ other 1-dimensional representations $W_2, \ldots, W_p$ that lie over $V_1$. Let $w_i \in W_i$ for $i=1,\ldots, p$ and let $v_1 \in V_1$. By \theref{irred} the glider representation
$$W_1 \oplus \ldots \oplus W_p \oplus P \supset KG_2(w_1,\ldots,w_p,v_1) \supset KZ(G)(w_1,\ldots,w_p,v_1) \supset K(w_1,\ldots,w_p,v_1)$$
is irreducible. However, since $KG_2(w_1,\ldots,w_p,v_1) = K(w_1,\ldots,w_p,v_1)$, this contradicts \lemref{dim} since $\dim(M) = p+1$. 

Next, if $[G:Z(G)] = p^3$, then there exists a chain $e \subset Z(G) \subsetneq G_2 \subsetneq G_3 \subsetneq G$. Let $V_1 \in \Irr(G_2)$ be a one-dimensional that has a one-dimensional $G$-representation $W_1$ lying over it. If $G/G_2 \cong C_{p^2}$ we can find $p^2$ non-isomorphic one-dimensional representations $W_1,\ldots, W_{p^2}$ lying over $V_1$. The same reasoning as above entails that there can not lie a $p$-dimensional irreducible $G$-rep over $V$. If $G/G_2 \cong C_p \times C_p$, the reasoning is different. Suppose in this case that we have a one-dimensional representation $W_1$ and a $p$-dimensional representation $P$ lying over $V_1$. Then we consider both $W_1$ and $P$ as $G_3$-reps. Since $G_3/G_2 \cong C_p$, we find $p$ non-isomorphic $G_3$-reps $W_1,\ldots, W_2, \ldots, W_p$ lying over $V_1$. By induction we know that $P$ can not be an irreducible $G_3$-rep, hence $P \cong P_1 \oplus \ldots \oplus P_p$ and, up to reordering, $P_1$ lies over $V_1$. Hence we can construct the following $FKG_3$ glider
$$P_1 \oplus W_1 \oplus \ldots \oplus W_p \supset KG_2(p_1,w_1,\ldots, w_p) \supset KZ(G)(p_1,w_1,\ldots, w_p) \supset K(p_1,w_1,\ldots, w_p),$$
with $w_i \in W_i$ for $i=1,\ldots, p$ and $p_1 \in P_1$. By construction we have that
$$KG_2(p_1,w_1,\ldots, w_p) = K(p_1,w_1,\ldots, w_p)$$ 
and \lemref{dim} entails that $\dim(P \oplus W_1 \oplus \ldots W_p) \leq \dim([G_3:G_2]) = p$, contradiction!

\begin{theorem}\thelabel{pgroup}
Let $G$ be a $p$-group and let $e \subset Z(G) \subset G_2 \subset \ldots \subset G_d = G$ be a maximal chain of normal subgroups. If $V, W \in \Irr(G)$ are such that they lie over the same irreducible representation of $G_{d-1}$, then $\dim(V) = \dim(W)$.
\end{theorem}
\begin{proof}
If $G_2 = G$, then $G = Z(G)$ and there is nothing to prove. For $G_3 = G$, it follows from the deductions above \lemref{dim}. If $G_4 = G$, suppose that $V$ and $W$ lie over some irreducible representation $U \in \Irr(G_3)$. If $\dim(U) = 1$, then $U_{|G_2} \in \Irr(G_2)$ and the result follows from our reasoning above. If $\dim(U) = p$, then $U \cong V_1 \oplus \ldots \oplus V_p$ as $G_2$-reps, hence $V$ and $W$ lie over some irreducible $V_i \in \Irr(G_2)$ and the result follows again by the reasoning above. Now, let $G_5 = G$ and $V,W$ lie over some $U \in \Irr(G_4)$. Suppose that $\dim(V) < \dim(W)$. Then $W_{|G_4}$ can not be irreducible, otherwise $W,V \in \Irr(G_4)$ would lie over a same irreducible $U' \in \Irr(G_3)$. So $W \cong W_1 \oplus \ldots \oplus W_p$ as $G_4$-reps and we may assume that $W_1 \cong U \cong V_{|G_4}$. But then it follows that $W \cong \Ind_{G_4}^G(W_1) \cong \Ind_{G_4}^G(U) \cong \Ind_{G_4}^G(V_{|G_4}) \supset V$, which shows that $W$ is not irreducible, contradiction.
The general result now follows from induction.
\end{proof}

The theorem shows that the Hasse diagram is not connected. For example, the Hasse diagram of $\{1\} \subset \{1,-1\} \subset \mathbb{Z}_4^j \subset Q_8$ of \exaref{q8} looks like\\

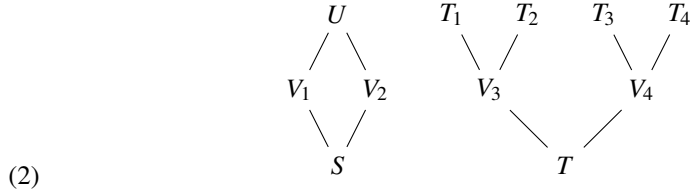
\begin{equation}\label{Hasse1}
\begin{tikzpicture}[scale=.5]
\node (U) at (-2,2) {$U$};
\node (T_1) at (1,2) {$T_1$};
\node (T_2) at (3,2) {$T_2$};
\node (T_3) at (5,2) {$T_3$};
\node (T_4) at (7,2) {$T_4$};
\node (V_1) at (-3,0) {$V_1$};
\node (V_2) at (-1,0) {$V_2$};
\node (V_3) at (2,0) {$V_3$};
\node (V_4) at (6,0) {$V_4$};
\node (S) at (-2,-2) {$S$};
\node (T) at (4,-2) {$T$};
\draw (U) -- (V_1) -- (S) -- (V_2) -- (U);
\draw (T_1) -- (V_3) -- (T_2);
\draw (V_3) -- (T) -- (V_4);
\draw (T_3) -- (V_4) -- (T_4);
\end{tikzpicture}
\end{equation}

Since a finite nilpotent group is the direct product of its Sylow subgroups, we immediately get the same result for finite nilpotent groups.

\begin{corollary}
Let $G = P_1\ldots P_n$ be a finite nilpotent group with Sylow subgroups $P_i$ of order $p_i^{a_i}$ and let $H \triangleleft G$ be a maximal normal subgroup. If $V,W$ lie over the same irreducible representation $U$ of $H$, then $\dim(V) = \dim(W)$.
\end{corollary}
\begin{proof}
Suppose that $H \triangleleft G$ is a maximal subgroup, that is, up to reordering, $H = Q_1P_2\ldots P_n$ with $[P_1:Q_1] = p_1$. Since $P_1 \triangleleft G$ is central, $U_{|P_1} = (U_1)^{p_2^{a_2}\ldots p_n^{a_n}}$ for $U = V,W$. This implies that $\dim(V) = p_2^{a_2}\ldots p_n^{a_n} \dim(V_1)$ and $\dim(W) = p_2^{a_2}\ldots p_n^{a_n}\dim(W_1)$. Since $V_1$ and $W_1$ are irreducible representations of the $p_1$-group $P_1$ lying over the same irreducible representation of $Q_1$, the result follows from \theref{pgroup}.
\end{proof}

As a short digression we propose a definition for the `character' or `generalized trace map' of a glider representation. As an application, we will consider the `generalized character table' of $Q_8$ and $D_8$, the notorious pair of groups when showing the shortcomings of character tables.\\

Let $\Omega \supset M \supset \ldots \supset M_n \supset \ldots$ be an $FKG$-glider representation. From general fragment theory, we may reduce to the situation where $B(M) = 0$ and $\el(M) = d$, i.e. $M_d \supsetneq M_{d+1} = 0$. For $j \leq i$ we have $KG_j M_i \subset M_{i-j}$, which we present in the following lower triangular matrix
$$A = \left(\begin{array}{ccccccc}
G_dM_d &&&&&&\\
G_{d-1}M_d & G_{d-1}M_{d-1} &&&&&\\
\vdots & \vdots & \ddots&&&&\\
G_iM_d & G_iM_{d-1} & \ldots & G_iM_i &&&\\
\vdots & \vdots & \ddots & \vdots & \ddots&&\\
G_1M_d & G_1M_{d-1} & \ldots &G_1M_i & \ldots & G_1M_1 & \\
M_d & M_{d-1} & \ldots & M_i & \ldots & M_1 & M_0
\end{array}\right)$$
The $i$-th row of $A$ consists of $G_i$-modules $G_iM_j$. We denote the associated characters by $\chi_{ij}$. We propose the following definition
\begin{definition}\deflabel{gentrace}
Let $\Omega \supset M$ be a $FKG$-glider representation with $B(M) = 0$ and $el(M) = d$, then the associated trace of $\Omega \supset M$ is the map
$$\chi_M: G \to M_{d+1}(K),~ g \mapsto \left(\begin{array}{ccccccc}
\chi_{dd}(g) &&&&&&\\
\chi_{d-1,d}(g) & \chi_{d-1,d-1}(g) &&&&&\\
\vdots & \vdots & \ddots&&&&\\
\chi_{i,d}(g) & \chi_{i,d-1}(g) & \ldots & \chi_{ii}(g) &&&\\
0 & 0 & \ldots & 0 & 0 &&\\
\vdots & \vdots & \ddots & \vdots & \vdots& \ddots&\\
0 & 0 & \ldots & 0 & \ldots & \ldots& 0
\end{array}\right),$$
when $g \in G_i - G_{i-1}$.
\end{definition}

Clearly $\chi_M(g) = \chi_M(g')$ if and only if $g \in c(g')$ and $g,g' \in G_i - G_{i-1}$. The matrix $\chi_M(e)$ lists the dimensions of the $G_jM_i$ and we have the following nice characterization

\begin{proposition}\proplabel{anti-diagonal}
Let $\Omega \supset M$ be a glider representation with respect to a finite group algebra filtration. Then $M$ is irreducible of essential length equal to the filtration length and with $B(M) = 0$ if and only if the matrix $\chi(e)$ is symmetric with respect to the anti-diagonal and has a 1 in the left most corner.
\end{proposition}
\begin{proof}
For $e \subset G$, the result is clear since we have irreducibility if and only if $M = GM_1$. We proceed by induction. Let $\Omega \supset M$ be an irreducible $FKG$-glider with $G_d = G$  such that $B(M) = 0$ and $\el(M) = d$, then \lemref{irr1} entails that $ M_1 \supset M_2 \supset \ldots$ is also irreducible, but since $\el(M) = d$, it is also an irreducible $FKG_{d-1}$-glider. Suppose that the result holds up to chains of length $d-1$ and let $G$ be some group with a chain of length $d$. Let $\Omega \supset M$ be an irreducible glider. Then $\Omega \supset M_1 \supset \ldots$ is an irreducible $FKG_{d-1}$-glider and the matrix
$$ \chi_M(e) = \left(\begin{array}{ccccccc}
\chi_{dd}(e) &&&&&\\
\chi_{d-1,d}(e) & \chi_{d-1,d-1}(e) &&&&\\
\vdots & \vdots & \ddots&&&\\
\chi_{i,d}(e) & \chi_{i,d-1}(e) & \ldots & \chi_{ii}(e) &&\\

\vdots & \vdots & \ddots & \vdots & \ddots\\
\chi_{0,d}(e) & \chi_{0,d-1}(e) & \ldots & \chi_{0,i}(e) & \ldots & \chi_{0,1}(e)
\end{array}\right)$$
is symmetric with respect to the anti-diagonal everywhere below the diagonal. We still have to check whether $\chi_{i,i}(e) = \chi_{d-i,d-i}(e)$ for all $0 \leq i \leq d$. These numbers correspond to the dimension of $KG_iM_i$ and $KG_{d-i}M_{d-i}$, but since $M$ is irreducible we have
$$KG_iM_i = KG_i KG_{d-i}M_d = KG_{\max\{i,d-i\}}M_d = KG_{d-i}KG_iM_d = KG_{d-i}M_{d-i}.$$
\end{proof}

For a fixed chain $e \subset G_1 \subset \ldots \subset G_d = G$, we can enlist all the irreducible gliders of essential length $d$ and calculate the `generalized character table' of the chain. For $p$-groups $G$ with $G_1 = Z(G)$ we can retrieve the classical character table of $G$ from the generalized table. Indeed, by looking at $\chi_M(e)$, we find all irreducibles with dimension vector $\dim(M) = (\dim(M_d),\dim(M_{d-1}),\ldots,\dim(M_0)) = (1,1, \ldots, 1, n)$. One should take care, for this does not imply that $M$ is an irreducible $G$-rep! Indeed, for $\{1\} \subset \{1,-1\} \subset \mathbb{Z}_4^j \subset Q_8$, we could take $T_1 \oplus T_2 \supset V_3 \supset T \supset T$. By looking at the $M$ with $\dim(M) = (1,\ldots,1)$  we recover all the 1-dimensional representations.Subsequenly, look at the $M$ with $\dim(M) = (1,1,1, \dim(M_{d-3}), \ldots, \dim(M_1), p^\alpha)$ then we have to look at $\chi_M(z)$ for $z \in G_2$. If $d=3$, \theref{pgroup} ensures that if $M$ is an irreducible $G$-rep, that the $G_2$-character of $M_{d-2} = KG_2M_d$ is different from the ones appearing in the irreducible gliders of dimension vector $(1,\ldots,1)$ already found. In other words, the irreducibility of $M$ as  $G$-rep is being detected by the character of the one-dimensional $G_2$-rep $M_{d-2}$.

\begin{example}
Consider now $\{1\} \subset \{1,-1\} \subset H \subset Q_8$ and $\{e\} \subset \{e,a^2\} \subset H' \subset D_8$, with $H$ and $H'$ two subgroups of order 4 containing the center. By the above reasoning we can localize the irreducible gliders with $M$ the unique simple 2-dimensional representation. The Hasse diagram is the same for both chains and its form \eqref{Hasse1} says that there are two irreducible gliders of dimension $(1,1,1,2)$ ending in $U$. If we then look at $\chi_M(h)$ for $h \in H$, we obtain the characters of the $H$-reps $V_1$ and $V_2$. From this information we can deduce whether $H$ must be isomorphic to $C_4$ or $V_4$. In the latter case we have $G = Q_8$.
\end{example}

\section{Decomposition groups}

In \cite{CVo1} a Clifford theory for glider representations was performed, leading to the existence of so-called building blocks with associated decomposition groups. In this section, we want to investigate a deeper connection between these decomposition groups.

Let $K$ be an algebraically closed field of ${\rm~char}K = 0$ and $G$ some finite group. Consider a finite group algebra filtration $FKG$ on $KG$ given by a chain of normal subgroups $1 \triangleleft G_1 \triangleleft \ldots \triangleleft G_d = G$. Let $H \triangleleft G$ be another normal subgroup and consider the filtration $FKH \subset FKG$ defined by the chain $1 \triangleleft H_1 \triangleleft \ldots \triangleleft H_d = H$ where $H_i = H \cap G_i$. Let $M$ be an irreducible $FKG$-glider representation of $\el(M) = d$ and $B(M) = 0$. For $0 \leq i \leq d$ we have the following group algebra filtrations
\begin{equation}\label{filtrations}
\begin{array}{cl}
(f1): &K \subset KH_1 \subset \ldots \subset KH_{i-1} \subset KH_{i-1}\\
(f2): &K \subset KH_1 \subset \ldots \subset KH_{i-1} \subset KH_{i}\\
(f3): &K \subset KG_1 \subset \ldots \subset KG_{i-1} \subset KG_{i}\\
(f4): &K \subset KG_1 \subset \ldots \subset KG_{i-1} \subset KG_{i-1}
\end{array}
\end{equation}
and it is clear that $M$ is a glider representation for each of these four filtrations. When performing the Clifford theory in \cite{CVo1} we only considered filtrations $(f2) \subset (f4)$ and obtained decomposition groups $H_{i} \subset G_i' \subset G_i$ depending on some building block $S$, which is a simple $H_{i-1}$-representation. What we want to do now, is to consider the inclusions
$$\begin{array}{ccc}
(f2) & \subset & (f4)\\
\cup & \dsubset & \cup \\
(f1) & \subset & (f3)
\end{array}$$
In total there are five inclusions, so for any building block $S$, say an irreducible $KH_{i}$-module, we obtain five decomposition groups:
$$\begin{array}{ccccc}
H_{i+1} & \subset & G_{i+1}' & \subset & G_{i+1}\\
\cup &&  & \dsubset~~~& \cup\\
H_i' & & H_i^2 && G_i''\\
\cup &~~~ \dsubset &&& \cup\\
H_i & \subset &G_i' & \subset & G_i 
\end{array}$$ 

We wonder whether there are some relations between these 5 decomposition groups. We assume that the chains on $G$ and $H$ are maximal, i.e. the subsequent factor groups are all simple and $H_{i+1}G_i = G_iH_{i+1} = G_{i+1}$.\\

Let us briefly recall from \cite{Cl} how the classical decomposition groups are constructed: Let $H \triangleleft G$ be a normal subgroup and let $V$ be some simple $G$-representation. If $V_H$ is no longer simple, there is some simple $H$-subrepresentation $W$ and we can find elements $g_2, \ldots, g_r \in G - H$ such that 
$$V = W \oplus g_2W \oplus \ldots \oplus g_rW.$$
All of the $g_iW$ are simple $H$-reps and of the same degree. Group together the ones that are isomorphic to yield a decomposition
$$V = R_1 \oplus \ldots \oplus R_m,$$
with $W \subset R_1$ and $m | r$. Moreover, all the $R_i$ are of equal dimension and the elements of $G$ permute the spaces $R_i$ among each other (\cite[Theorem 2]{Cl}). The decomposition group $H \subset G' \subset G$ is then the group of elements $g$ that leave $R_1$ invariant, i.e. $gR_1 = R_1$. 

\begin{lemma}
The decomposition group $G'$ is exactly the group of elements $g \in G$ such that $gW \cong W$ as $H$-modules.
\end{lemma}
\begin{proof}
If $g \in G'$, then (up to reordering the $g_i$ from above) $gR_1 = g(W \oplus g_2W \oplus \ldots \oplus g_{r'}W) = R_1 = W \oplus g_2W \oplus \ldots \oplus g_{r'}W$ (where $r' = r/m$). Hence $gR_1 = gW \oplus gg_2W \oplus \ldots \oplus gg_{r'}W$ and because the decomposition into $H$-components is unique, we must have that $gW \cong g_iW$ for some $1 \leq i \leq r'$. But by construction, $g_iW \cong W$. Conversely, suppose that $gW \cong W$ as $H$-modules, then for all $1 \leq i \leq r'$, $gg_iW \cong gW \cong W$. Because $gR_1 = R_j$ for some $1 \leq j \leq m$, $j$ must be 1, since the irreducible $H$-components of $gR_1$ are all isomorphic to $W$.
 \end{proof} 
The previous lemma allows to obtain a first result
\begin{proposition}\proplabel{4.1}
In the situation above, we have the following
\begin{itemize}
\item $H_i^2 = G_i \Leftrightarrow G_i' = G_i {\rm~and~}H_i' = H_i;$
\item $H_i^2 = H_{i+1} \Leftrightarrow G_i' = H_i {\rm~and~}H_i' = H_{i+1};$
\item $H_i^2 = G_{i+1} \Leftrightarrow G_i' = G_i {\rm~and~} H_i' = H_{i+1};$
\end{itemize}
\end{proposition}
\begin{proof}
Observe that $H_i' G_i' \subset H_i^2$. 
Suppose that $G_i' = G_i$ and $H_i' = H_{i+1}$, then $G_{i+1} = G_iH_{i+1} \subset H_i^2$. Conversely, if $H_i^2 = G_{i+1}$, then for all $g \in G_1,~gS \cong S$ so $G_i' = G_i$ and similarly we have $H_i' = H_{i+1}$. This covers the last equivalence. The first two equivalences are symmetric, so we only prove the first one. If $H_i^2 = G_i$ then automatically $G_i' = G_i$ and $H_i'$ must be $H_i$ by the third equivalence and our maximality condition. Conversely, we have $H_i'G_i' = G_i \subset H_i^2$ and there are no proper subgroups in between $G_i$ and $G_{i+1}$ so $H_i^2 = G_i$.
\end{proof}

The previous proposition indicates that the two smaller decomposition groups $G_i'$ and $H_i'$ almost determinate the group $H_i^2$. The interesting situation we have to study deeper is when $H_i^2$ equals $H_i$ or some group $F$ lying strictly between $H_i \subsetneq F \subsetneq G_{i+1}$. In many cases however, this situation does not arise:

\begin{lemma}\lemlabel{center}
In the situation above, if $H_i \subset Z(G_{i+1})$ then the three decomposition groups $H_i',~G_i'$ and $H_i^2$ are all maximal, that is, they are $H_{i+1},~G_{i}$ and $G_{i+1}$ respectively.
\end{lemma}
\begin{proof}
Let us proof this for $H_i'$, the other two cases being completely analogous. Decompose $KH_{i+1}S = S \oplus h_2S \oplus \ldots \oplus h_rS$ into simple $H_i$-modules for some $h_j\in H_{i+1} - H_i$. Since $H_i \subset Z(G)$, $H_i$ commutes with the appearing $h_j$ so for $h \in H_i$ we have $h \cdot h_jS = hh_jS = h_j(h\cdot S)$, which shows that $S$ and $h_jS$ are isomorphic $H_i$-modules. Hence $KH_{i+1}S = R_1$ and it follows that $H_i' = H_{i+1}$.
\end{proof}

\begin{example}\exalabel{q8}
In \cite{CVo1} we looked at the following  graph of groups
$$\begin{array}{ccc}
\mathbb{Z}_4^j = \{1,j,-1,-j\} & \triangleleft& Q_8 = <-1,i,j | i^2 = j^2 =  -1, ij = -ji >\\
\triangledown & & \triangledown\\
\mathbb{Z}_2 = \{1,-1\} & \triangleleft & \mathbb{Z}_4^i = \{1,i,-1,-i\}
\end{array}$$
and we studied the following irreducible glide representation
$$ \Omega = M = U \oplus T_3 \oplus T_2 \supseteq V^{-i} \oplus T_3 \oplus T_2 \supseteq \Delta \supset 0 \supset \ldots,$$
where the $T_i$ denote the four 1-dimensional $Q_8$-representations, given by
$$\begin{array}{c}
T_1 : i \mapsto 1,~ j \mapsto 1 \\
T_2 : i \mapsto -1,~ j \mapsto 1 \\
T_3 : i \mapsto 1,~ j \mapsto -1 \\
T_4 : i \mapsto - 1,~ j \mapsto - 1 
\end{array}$$
and $U$ is the simple 2-dimensional representation
$$ U: i \mapsto \begin{pmatrix} i & 0 \\ 0 & -i \end{pmatrix},~ j \mapsto \begin{pmatrix} 0 & - 1 \\ 1 & 0 \end{pmatrix}$$
Under base change 
$$\begin{array}{c}
e_1 = f_1 + if_2\\
e_2 = f_1 - i f_2
\end{array}$$
we get $U = V^{-i} \oplus V^i$, where $V^{i}$ is the simple $\mathbb{Z}_4$-representation defined by $ j \mapsto i$ and similarly for $V^{-j}$. Observe that $M$ is indeed irreducible by \theref{irred}.
We showed that $\{\mathbb{C}e_1, \mathbb{C}(t_3+t_2)\}$, where $t_3 \in T_3, ~t_2 \in T_2$,  is a minimal set of building blocks and after some calculations we arrive at the following decomposition groups:

 $$\mathbb{C}e_1:~ \begin{array}{ccc}
 & Q_8 &\\
 \mathbb{Z}_4^i & Q_8 &  \mathbb{Z}^j_4\\
 &  \mathbb{Z}_4^j &
 \end{array} \quad \mathbb{C}(t_3+t_2): ~ \begin{array}{ccc}
 & Q_8 &\\
  \mathbb{Z}_4^i & Q_8 & Q_8\\
  &  \mathbb{Z}^j_4 & 
  \end{array}$$
  
$H_1 = \{1,-1\} = Z(Q_8)$, so the three lower decogroups are fixed. But we see that for the two others, there are some differences. 
\end{example}

Let us look at the following situation for our diagram

\begin{equation}\label{situatie1}
\begin{array}{ccccc}
H_{2} & \subset & H_2 & \subset & G_{2}\\
\cup &&  & \dsubsetneq~~~& \cup\\
H_1 & & F && G_i''\\
\cup &~~~ \dsubsetneq &&& \cup\\
H_1 & \subset &H_1 & \subset & G_1 
\end{array} 
\end{equation}
\\
Start with a simple $H_1$-module $S$ in some irreducible glider representation $M$ and decompose $KH_2S = S \oplus h_2S \oplus \ldots \oplus h_rS$. Since $H_1' = H_1$, we have that $h_iS = R_i,~ i=1,\ldots,~r$, whence $[H_2:H_1] = r$. We do the same for 
\begin{eqnarray*}
KG_2S &=& KH_2S \oplus g_2KH_2S \oplus \ldots \oplus g_mKH_2S\\
&=& S \oplus h_2S \oplus \ldots \oplus h_rS \oplus g_2h_2S \oplus \ldots \oplus g_mh_rS.
\end{eqnarray*}
Again, since $G_2' = H_2$ we have $[G_2:H_2] = m$ and as $H_1^2 = F$ is proper, there is some $m' \leq m$ such that (up to reordering) $R^{H_1^2}_1 = S \oplus g_2h_{j(2)}S \oplus \ldots \oplus g_{m'}h_{j(m')}S$. Hence we have that 
$$[G_2:F] = \frac{mr}{m'} = \frac{[G_2:H_2][H_2:H_1]}{m'} = \frac{[G_2:H_1]}{m'}.$$

Assume now that $G_2$ is a $p$-group. Then $p = [G_2:F] = \frac{p^2}{m'}$, hence $m' = p = m$. Moreover, since $G_2' = H_2$, we have that $g_iH_2 \cap g_jH_2 = \emptyset$ for $1 \leq i \neq j \leq p$, so $\{H_2,~g_2H_2,\ldots,~g_pH_2\}$ is a full set of left cosets of $H_2$ in $G_2$. We have that $[F:H_1] = p$, hence we can find $f_i$, $i=1,\ldots, p$ forming a full set of representatives. Write $f_i = g_{\alpha(i)}h_{\beta(i)}$. If $\alpha(i) = \alpha(j)$ for some $i \neq j$, then $h_{\beta(i)}S \cong h_{\beta(j)}S$, or $h_{\beta(i)}h^{-1}_{\beta(j)} \in H_1$. However, $f_i^{-1}f_j  = h^{-1}_{\beta(i)}h_{\beta(j)} \notin H_1$, a contradiction. Therefore all the $g_iKH_2S$ have one component isomorphic to $S$, whence they have the same decomposition as $H_1$-modules (since $H_2$ is normal in $G_2$). So we have proven

\begin{proposition}\proplabel{4.2}
Let $G_2$ be a $p$-group and $S$ some building block such that its associated decomposition graph takes the form \eqref{situatie1}. Then $H_1$ has at least $p$ non-isomorphic irreducible representations of degree $\dim(S)$, call these $S_1,\ldots,~S_p$. Moreover, there are at least $p$ non-isomorphic $H_2$-modules that decompose into $\oplus_{i=1}^p S_i$.
\end{proposition}

Next, we consider the situation

\begin{equation}\label{situatie2}
\begin{array}{ccccc}
H_{2} & \subset & G_2 & \subset & G_{2}\\
\cup &&  & \dsubsetneq~~~& \cup\\
H_1 & & F && G_1''\\
\cup &~~~ \dsubsetneq &&& \cup\\
H_1 & \subset &H_1 & \subset & G_1 
\end{array} 
\end{equation}
\\
To begin with, $H_1' = H_1$ means that we can find $p$-elements $h_i \in H_2 - H_1$ such that $KH_2S = S \oplus h_2S \oplus \ldots \oplus h_pS$. In fact, since $|H_2/H_1| = p$, the factor group is cyclic and we can choose the elements $h_i$ such that $h_i = h_2^{i-1}$ for $i=2, \ldots, p$ and $h_1 = e$. There are two situations when $G_2' = G_2$ occurs. Indeed, either $KG_2S = KH_2S$, saying that $KH_2S$ is already a $G_2$-representation, or either $KG_2S = KH_2S \oplus g_2KH_2S \oplus \ldots \oplus g_pKH_2S$ for $p$ elements $g_i \in G_2 - H_2$. In the first situation however, $KG_2S = S \oplus h_2S \oplus \ldots \oplus h_pS$ and $H_1 \subsetneq F \subsetneq G_2$ contradicts $H_1'$ being equal to $H_1$. So 
$$KG_2S = (S \oplus h_2S \oplus \ldots \oplus h_pS) \oplus g_2(S \oplus h_2S \oplus \ldots \oplus h_pS) \oplus \ldots$$
Clearly, all the $g_iKH_2S$ have the same decomposition into $H_1$-modules (as was the case in situation \eqref{situatie1}), so the only difference between situations \eqref{situatie1} and \eqref{situatie2} lies in the fact the $g_iKH_2S$ are isomorphic $H_2$-modules or not. Let us treat both cases simultaneously and see what happens.\\

For some $\alpha: \{1,\ldots,~p \} \to \{1,\ldots,p\}$, we have:
$$S \cong g_2h_{\alpha(2)}S \cong \ldots \cong g_ph_{\alpha(p)}S.$$
 Since the factor group $G_2/H_1$ is of order $p^2$, it is isomorphic to either $C_{p^2}$ or $C_p \times C_p$. However, clearly $F/H_1, H_2/H_1$ and $G_1/H_1$ are different subgroups of order $p$ in $G_2/H_1$, so since $C_{p^2}$ has only one subgroup of order $p$, $G_2/H_1$ must be isomorphic to $C_p \times C_p$.\\
 
Of course we have that $\alpha(1) = 1$. Suppose that $\alpha(i) = 1$ for some $2 \leq i \leq p$, then $g_i = g_2^{i-1} \in F$, whence $g_j \in F$ for all $1 \leq j \leq p$ and it follows that $\alpha$ maps to $1$. If $\alpha(i) \neq 1$ for all $2 \leq i \leq p$, then $\alpha$ must be bijective, otherwise $g_2^i \in F$ for some $i$, meaning that $\alpha(i) = 1$. Up to changing $g_2$ by $g_2h_{\alpha(2)}$ we may assume that $\alpha = 1$ and we see that $F/H_1 = <\ov{g_2}>$. Going the other way, that is, via $G_1$, we find a $z_2 \in G_2 - G_1$ such that $F/H_1 = <\ov{z_2}>$. Hence $z_2H_1 = g_2H_1$ and we have that $z_2 \notin H_2, g_2 \notin G_1$.

\begin{proposition}\proplabel{4.3}
Let $G_2$ be a $p$-group and $S$ some building block such that its associated decomposition graph takes the form \eqref{situatie1} or \eqref{situatie2}. Then $G_2/H_1 \cong C_p \times C_p$ and we find elements $g, z \in G_2 - (G_1 \cup H_2)$ such that $F/H_1 = <g > = <z>$. 
\end{proposition}

\begin{example}
Assume $p = 5$ and $H_1$ some $5$-group. Consider $G_2 = H_1 \rtimes (C_5 \times C_5)$, a semi-direct product defined by some group morphism $\varphi: C_5 \times C_5 \to \Aut(H_1)$, i.e. 
$$ (h,a^i, a^j) \cdot (h',a^k,a^l) = (h \varphi((a^i,a^j))(h'), a^{i+k},a^{j+l}),$$
where $C_5 = <a>$. then we look at

$$\begin{array}{ccccc}
H_1 \rtimes <(a^2,a)> & \subset & G_2' & \subset & H_1 \rtimes (C_5 \times C_5)\\
\cup &&  & \dsubsetneq~~~& \cup\\
H_1 & & F && G_1''\\
\cup &~~~ \dsubsetneq &&& \cup\\
H_1  & \subset &H_1 & \subset & H_1 \rtimes <(a,a^2)>
\end{array} $$
Suppose that if we take $g = g_2 = (e,a^3,a^2)$ that $\alpha = 1$. Then $F/H_1 = <\ov{g_2}>$, or $F = H_1 \sqcup H_1g \sqcup \ldots \sqcup H_1g^{p-1}$. Define a map
$$f: F \to H_1 \rtimes <\ov{g}>,~ f \mapsto (h,\ov{g}^i),$$
if $f = hg^i$ and where the semi-direct product structure $H_1 \rtimes <\ov{g}>$ is defined by $\varphi_{|<(a^3,a^2)>}$. We calculate
\begin{eqnarray*} 
f(hg^i)f(h'g^j) &=& (h,\ov{g}^i)(h',\ov{g}^j) \\
&=& (h\varphi(\ov{g}^i)(h'), \ov{g}^{i+j})\\
f(hg^ih'g^j) &=& f(h\varphi(g^i)(h')g^{i+j}) \\
&=& (h\varphi(\ov{g}^i)(h'), \ov{g}^{i+j}),
\end{eqnarray*} 
which shows that $f$ is a group morphism, which is easily seen to be surjective, hence bijective. Moreover, 
$$H_1 \rtimes <\ov{g}> = H_1 \rtimes F/H_1 \subset H_1 \rtimes G_2/H_1 =H_1 \rtimes G_2/H_2 \times G_2/G_1,$$
so $F = H_1 \rtimes <(a^3,a^2)>$. 
\end{example}

Observe that, in the previous example, if we drop the semi-direct product with $H_1$, then everything is abelian and all decomposition groups would be maximal by \lemref{center}. In fact, the only possibility to have situation \eqref{situatie1} or \eqref{situatie2} is when $G_2 \cong H_1 \rtimes (C_p \times C_p)$. Indeed

\begin{theorem}\thelabel{4.4}
Let $G_2$ be a $p$-group and $S$ some building block such that its associated decomposition graph takes the form \eqref{situatie1} or \eqref{situatie2}. Then $G_2 \cong H_1 \rtimes (C_p \times C_p)$ where the semi-direct product is not a direct product. Moreover $F = H_1 \rtimes <\ov{g}> = H_1 \rtimes <\ov{z}>$ for some $g, z \in G_2 - (H_2 \cup G_1)$ such that $F \cong H_1 \rtimes (<\ov{g},\ov{z}>) \subset H_1 \rtimes G_2/H_2 \times G_2/G_1$.
\end{theorem}
\begin{proof}
By our observations above, we can find elements $z, g \in G_2 - (H_2 \cup G_1)$ such that $F/H_1 = <\ov{g}> = <\ov{z}>$. This shows that $F = H_1 \rtimes <\ov{g}> = H_1 \rtimes <\ov{z}>$. We have that $G_2/H_1 \cong G_2/H_2 \times G_2/G_1$ and the isomorphism maps $\ov{g}$ to $(\ov{g},\ov{z})$. Hence
$$F = H_1 \rtimes <\ov{g}> \cong H_1 \rtimes <(\ov{g},\ov{z})> \subset H_1 \rtimes (<\ov{g}> \times <\ov{z}>) = H_1 \rtimes (G_2/H_2 \times G_2/G_1) \cong G_2$$
This also shows that the semi-direct product is not a direct product, otherwise the decomposition group $H_1^2 = G_2$, because in this case $H_1$ commutes with $ \{e\} \times <\ov{g}> \times <\ov{z}>$.
\end{proof}

\begin{corollary}
Let $G_2$ be a $p$-group and $S$ some building block such that its associated decomposition graph takes the form \eqref{situatie1} or \eqref{situatie2}. Then $H_1 \neq C_p$.
\end{corollary}
\begin{proof}
The order of the automorphism group $|\Aut(H_1)| = \varphi(p) = p-1$, so there are no non-trivial group morphism $f: C_p \times C_p \to \Aut(H_1)$, so $G_2$ woud not be a semi-direct product.
\end{proof}

The last situation we have to consider is when $H_i^2 = H_i$, i.e. 

\begin{equation}\label{situatie3}
\begin{array}{ccccc}
H_{2} & \subset & G_2' & \subset & G_{2}\\
\cup &&  & \dsubset~~~& \cup\\
H_1 & & H_1 && G_1''\\
\cup &~~~ \dsubset &&& \cup\\
H_1 & \subset &H_1 & \subset & G_1 
\end{array} 
\end{equation}

We have the decomposition of $KG_2S$ into $p^2$ irreducible $H_1$-representations. This shows for example that $p^2\dim(S) \leq |H_1|$. In particular we have that $KG_2S \supsetneq KH_2S$ and $KG_2S \supsetneq KG_1S$. So we can write
$$KG_2S = KH_2S \oplus g_2 KH_2S \oplus \ldots \oplus g_p KH_2S,$$
$$KG_2S = KG_1S \oplus z_2 KG_1S \oplus \ldots \oplus z_p KG_1S,$$
for some $g_i \in G_2 - H_2$ and $z_i \in G_2 - G_1$. If for example $KH_2S \cong g_2KH_2S$ as $H_2$-reps, then they would have the same decomposition into $H_1$-reps, which is a contradiction. Hence we arrive at

\begin{proposition}
Let $G_2$ be a $p$-group and $S$ some building block such that its associated decomposition graph takes the form \eqref{situatie3}. Then $G_2' = H_2$ and $G_1'' = G_1$.
\end{proposition}

\section{Nilpotent groups of order $p^kq^l$}

It is well-known that a nilpotent group is the direct product of its Sylow subgroups, hence we consider the group $GH$, where $G$ is a $p$-group and $H$ is a $q$-group. If the order of $GH$ is $pq$ and $p-1$ does not divide $q$, then the group is cyclic, and it is just isomorphic to $C_p \times C_q$, which is not interesting to consider chains of subgroups. So assume that $|GH| = p^kq^l$ with $k, l > 1$. We consider the following chains of subgroups
$$\begin{array}{ccccccccc}
e & \subset & G_1 & \subset & \ldots & \subset & G_{d-1} & \subset & GH\\
&&&&&&& \dsubset & \cup\\
&&&&&& G_{d-1}H_{d-1} && H_{d-1}\\
&&&&& \dsubset & && \cup\\
&&&& \iddots &&&& \vdots\\
&&& \dsubset &&&&& \cup\\
&& G_1H_1 &&&&&& H_1\\
& \dsubset &&&&&&& \cup\\
e &&&&&&&& e
\end{array}$$

The Frobenius divisibility theorem (\cite[Theorem 4.16]{Etingof}) states that the dimension of an irreducible representation $V$ divides the order of the group , so we enlist the irreducible representations of $GH$ by
$$V_i ~ (i \in I_1)~ P_i~ (i \in I_2), ~Q_i~ (i \in I_3),~W_i~(i \in I_4),$$
where $V_i$ is 1-dimensional, $P_i$ is $p^\alpha$-dimensional, $Q_i$ is $q^\beta$-dimensional and $W_i$ is $p^\alpha q^\beta$-dimensional. Let $\Omega \supset M$ be an irreducible glider of essential length $d$. \theref{irred} shows that 
\begin{equation}\label{dec}
M = \bigoplus_i V_i \oplus \bigoplus_{i} P_i^{m_i} \oplus \bigoplus_i Q_i^{k_i} \oplus \bigoplus_i W_i^{l_i}
\end{equation}

\begin{lemma}
If $M_{H}$ is an irreducible $KH$-glider, then the $m_i$ in the decomposition \eqref{dec} are all zero.
\end{lemma}
\begin{proof}
In the decomposition of the $GH$-module $M$ as an $H$-module, there appears the decomposition of $P_i$ into $H$-modules (if $n_i \neq 0$). Since $P_i$ is $p^\alpha$-dimensional, its decomposition into simple $H$-reps must be
$$P_i = \oplus_{j \in J} U_j^{n_j},$$
where all the $U_j$ are 1-dimensional. In fact, all the $U_j$ must be isomorphic, since we have that $KGHU_1 = P_i$ and in the procedure of determining the decomposition group, we can take for elements in $GH - H$ elements of the form $(g,e)$. Since $(g,e)$ commutes with $e \times H$, the decomposition group is the whole group $GH$, whence alle the $U_i$ are isomorphic. But this contradicts with \theref{irred} so no factors $P_i$ appear in \eqref{dec}.
\end{proof}

\begin{lemma}\lemlabel{W}
If $M_{H}$ is an irreducible $KH$-glider and $W_i^{l_i}$ appears in \eqref{dec} with $\dim(W_i) = p^\alpha q^\beta$ then $l_i p^\alpha \leq q^\beta$.
\end{lemma}
\begin{proof}
In the decomposition of $M$ as an $H$-module there appears the decomposition of $W_i^{l_i}$ into $H$-modules. By the same argument in the proof of the previous lemma, the decomposition group must be $GH$, hence $W = U^{k}$ for some irreducible $H$-rep and some $k \in \mathbb{N}$. $U$ cannot be $1$-dimensional, because then $M_H$ would not be irreducible. Thus $\dim(U) = q^j$ for some $j$ and $k =p^\alpha q^{\beta -j}$. In particular we have that $\beta \geq j$. \theref{irred} then entails that $ l_i k = l_i p^\alpha q^{\beta - j} \leq q^j$. Hence
$$l_ip^\alpha \leq q^{2j - \beta} = q^{2j - 2\beta} q^\beta  \leq q^\beta,$$
where the last inequality follows since $j - \beta \leq 0$.
\end{proof}

\begin{proposition}\proplabel{irrGH}
Let $\Omega \supset M$ be an irreducible $KGH$ glider of essential length $d$, If $M_H$ and $M_G$ are still irreducible, then $M = \oplus_{i} V_i$.
\end{proposition}
\begin{proof}
Irreducibility as $KH$-, resp. $KG$-glider shows that no $P_i$'s, resp. $Q_i$'s appear. Suppose that some $W_i^{l_i}$ of dimension $p^\alpha q^\beta$ appears. From \lemref{W} applied for $KH$ and $KG$ shows that 
$$l_i p^\alpha \leq q^\beta \leq \frac{p^\alpha}{l_i},$$
hence $l_i =1$, but then $p^\alpha = q^\beta$, a contradiction. 
\end{proof}

Recall that the irreducible representations of a nilpotent group $GH$ of order $p^\alpha q^\beta$ arise as tensor products $U \ot V$ of irreducible representations of $G$, resp. $H$. We clearly have that for an irreducible $GH$-representation $W$ we have that $W_G$ and $W_H$ are irreducible if and only if $W$ is 1-dimensional. \propref{irrGH} gives one direction of this statement in the situation of glider representations. Let $M = \oplus_i V_i \supset \ldots \supset Ka$ be an irreducible $KGH$-glider of $\el(M) = d$. If some of the appearing $V_i$, say $V_1$ and $V_2$, are isomorphic as $G$-modules, then $M_G$ is definitely no longer irreducible by \theref{irred}. So we have the glider analogue 

\begin{proposition}\proplabel{irrGH2}
Let $\Omega \supset M$ be an irreducible $KGH$ glider of essential length $d$. Then $M_H$ and $M_G$ are irreducible if and only if $M = \oplus_i V_i$, where all the $V_i$ are non-isomorphic both as $G$- and as $H$-modules.
\end{proposition}

Suppose now that $M = \oplus_i V_i$, where the $V_i$ are non-isomorphic 1-dimensional $GH$-reps, but we drop the condition of $M_H$ and $M_G$ remaining both irreducible gliders. We can consider $\Omega \supset M$ as an $FKG_{d-1}H_{d-1}$-glider with the second filtration of \eqref{filtrations} on $G_{d-1}H_{d-1}$. Then either $M$ remains irreducible, which says that $K_{d-1}H_{d-1}a = \oplus_i V_i$, or either $ K_{d-1}H_{d-1}a \subsetneq \oplus_i V_i$. The first situation occurs exactly when all the $V_i$ remain non-isomorphic as $G_{d-1}H_{d-1}$-reps. If we are in the latter case, suppose that $V_1 \cong V_2$ as $G_{d-1}H_{d-1}$. This means that, up to renumbering, we can take $S \cong V_1$ as a $G_{d-1}H_{d-1}$-building block. However, looking at the decomposition group $G_{d-1}H_{d-1} \subset F \subset GH$ does not give us more information. Indeed, in the decomposition of $M$ we have $S$ and some conjugate $(g,h)S$ for some $(g,h) \in GH - G_{d-1}H_{d-1}$, but these are equivalent as $G_{d-1}H_{d-1}$-modules by assumption. Hence $F = GH$. So we have to take a different approach here. Rewrite 
$$M = \oplus_i U_i \ot U'_i,$$
where the $U_i$, resp. $U'_i$ are 1-dimensional $G$-, resp. $H$-reps. The following statements are obvious
$$\begin{array}{c}
M_G {\rm~is~irreducible~} \Leftrightarrow {\rm~all~the~appearing~} U_i {\rm~are~non}-{\rm isomorphic};\\
M_H {\rm~is~irreducible~} \Leftrightarrow {\rm~all~the~appearing~} U'_i {\rm~are~non-isomorphic}.
\end{array}$$

\begin{example}
If $M = U_1 \ot U_1' \oplus U_1 \ot U_2' \supset \ldots \supset Ka$, then 
$$M_G \cong U_1^{\oplus 2}, \quad M_H \cong U'_1 \oplus U'_2.$$
One verifies that the decomposition groups in both cases are just $F = GH$. The difference, however, lies in the number of building blocks! Indeed, for $M_G$, we have that $KGa \cong U_1 \subset U_1^{\oplus}$, where the embedding is of diagonal type. Hence there is only one building block $S = KGa$ and $M_G = S \oplus (e,h)S$. On the other hand, $KHa \cong U'_1 \oplus U'_2$, so we have two buidling blocks $S_i = U'_i, ~i=1,2$.
\end{example}

In fact, the previous example easily generalizes to 
\begin{lemma}
Let $M = \oplus_i V_i$ be an irreducible $FKGH$-glider, then the number of building blocks as $FKG$-, resp. $FKH$-glider corresponds to the number of non-isomorphic $G$-, resp. $H$-modules among the appearing $V_i$ in the decomposition of $M$.
\end{lemma}

To any irreducible $FKGH$-glider of $\el = d$ and of the form $M = \oplus_i V_i$ we now associate a triple $(a,b,c)$, where
$$\begin{array}{l}
a = {\rm~number~of~} GH{\rm-compontents~of~} M\\
b = {\rm~number~of~building~blocks~as~} KG-{\rm glider},\\
c = {\rm~number~of~building~blocks~as~} KH-{\rm glider}.
\end{array}$$

\begin{proposition}\proplabel{a=bc}
In the situation above, the glider $M$ is the tensor product $M = M_1 \ot M_2$ of an $FKG$- and an $FKH$-glider if and only if the associated triple $(a,b,c)$ satisfies $a = bc$.
\end{proposition}
\begin{proof}
Suppose that $M = M_1 \ot M_2$. If $M_1 = U_1^{\oplus n_1} \oplus \ldots \oplus U_b^{\oplus n_b}$  and $M_2 = (U'_1)^{\oplus m_1} \oplus \ldots \oplus (U'_c)^{\oplus m_c}$ with at least one of the $n_i >1$ or $m_1 > 1$, say $n_1 > 1$, then $M$ would contain two distinct isomorphic components  $U_1 \ot U_1'$, which contradicts the irreducibility of $M$ since these components are 1-dimensional. Hence $M_1 = U_1 \oplus \ldots \oplus U_b$ and $M_2 = U'_1 \oplus \ldots \oplus U'_c$, from which $a = bc$ follows. Conversely, assume $a = bc$. Up to reordering, we can write
$$M = (V_1 \oplus \ldots \oplus V_b) \oplus (V_{b+1} \oplus \ldots \oplus V_{2b}) \oplus \ldots \oplus (V_{b(c-1)+1} \oplus \ldots \oplus V_{bc}),$$
such that for every $0 \leq i \leq b$: $V_{jb + i} \cong V_{kb + i}$ as $G$-reps for all $0 \leq j,k \leq c-1$. Hence we have that 
$$ U_i \ot U'_{jb+i} \cong V_{jb + i} \cong V_{kb+i} \cong U_i \ot U'_{kb+i}$$
and it follows that 
$$M \cong U_1 \ot (U'_1 \oplus U'_{b +1} \oplus \ldots U'_{b(c-1) +1}) \oplus \ldots \oplus U_b \ot (U'_b \oplus U'_{2b} \oplus \ldots \oplus U'_{bc}).$$
If for example $U'_1 \cong U'_{b+1}$ as $H$-reps, then $V_1 \cong V_{b+1}$ as $GH$-rep which contradicts the irreducibility of $M$. Since only $c$ non-isomorphic $H$-reps occur in the above decomposition we have that 
$$U'_1 \oplus U'_{b +1} \oplus \ldots U'_{b(c-1) +1} \cong \ldots \cong U'_b \oplus U'_{2b} \oplus \ldots \oplus U'_{bc}$$
as $H$-reps, which entails that 
$$M \cong (U_1 \oplus \ldots \oplus U_b) \ot (U'_1 \oplus \ldots \oplus U'_{b(c-1)+1}).$$
\end{proof}
\begin{corollary}
If $\Omega_1 \supset M_1$, resp. $\Omega_2 \supset M_2$ are $KG$-, resp. $KH$-gliders such that all the appearing components of $\Omega_1$ and $\Omega_2$ are 1-dimensional, then we have that $(\Omega_1 \supset M_1) \ot (\Omega_2 \supset M_2)$ is an irreducible $FKGH$-glider if and only if $\Omega_1 \supset M_1$ and $\Omega_2 \supset M_2$ are irreducible.
\end{corollary}

If we drop the condition that the appearing components be all 1-dimensional, then we still have that irreducibility as $KG$- and as $KH$-glider implies irreducibility of the tensor product. The converse is no longer true. 
\begin{example}\exalabel{tensor}
Consider the trivial chain $K \subset KGH$. Let $P$ be a $p$-dimensional simple $G$-rep with $p, p' \in P$ linear independent. If $U$ is a one-dimensional $H$-rep then $M \supset Ka$ with $M = P \ot U_1 \oplus P \ot U_1$, $a = p \ot u + p' \ot u$ is an irreducible $FKGH$-glider. We can write 
$$ (M \supset Ka) = (P \supset K(p+p')) \ot (U^{\oplus 2} \supset Ku),$$
but $U_1^{\oplus 2} \supset Ku$ is not irreducible.
\end{example}

As another corollary we have a nice characterization of abelian nilpotent groups in terms of glider representations:
\begin{theorem}\thelabel{char}
Let $GH$ be a nilpotent group of order $p^\alpha q^\beta$. TFAE
\begin{enumerate}
\item $GH$ is abelian;
\item Every irreducible $FKGH$ glider $M$ is isomorphic to the tensor product $M_1 \ot M_2$ of an $FKG$- and an $FKH$-glider if and only if the associated triple $(a,b,c)$ of $M$ satisfies $a = bc$.
\end{enumerate}
\end{theorem}
\begin{proof}
If $GH$ is abelian, then all irreducible representations are 1-dimensional and the result follows from \propref{a=bc}. Conversely, suppose that $GH$ is not abelian, so there exists some $p$-dimensional irreducible representation $P \in \Irr(G)$ (up to interchanging the role of $G$ and $H$) and from \exaref{tensor} we get a counterexample.
\end{proof}

Suppose that $d = 3$, then \theref{pgroup} allows to extend the result of \propref{a=bc} when $\dim(M_1) = 1$.

\begin{proposition}
Let $d=3$ and $\Omega \supset M \supset M_1 \supset M_2 \supset Ka$ be an irreducible $FKGH$-glider with $M_1 = Ka$, then $M$ is isomorphic to the tensor product $N_1 \ot N_2$ of an $FKG$- and an $FKH$-glider if and only if the associated triple $(a,b,c)$ satisfies $a = bc$.
\end{proposition}
\begin{proof}
\theref{pgroup} entails that any irreducible $G_2$-representation (which is 1-dimensional as $G_2$ is abelian) lifts to either a $p$-dimensional simple $G$-rep or to $p$ non-isomorphic 1-dimensional simple $G$-reps. Moreover, since $p^2 + (|G_2|-p)p = p|G_2| = |G|$, we know that $G$ has only one $p$-dimensional simple, with $p$ non-isomorphic $G_2$-components. We can do the same for $H$ (replacing $p$ by $q$) and since $M_1 \cong V \ot U$, with $V \in \Irr(G_2), U \in \Irr(H_2)$ ($M_1 = KG_2H_2a = Ka$, with $G_2H_2$ abelian), we have four cases. We introduce the following notation: if $W$ is an irreducible $G_2$- (or $H_2$-)rep that lifts to a $p$- (or $q$-)dimensional simple, we denote it by $\bigvee$, if $W$ lifst to $p$ one-dimensional simples, we denote it by $\bigwedge$. So we have the following four cases:
$$\begin{array}{cc}
(\bigwedge, \bigwedge): & {\rm result~follows~from~} \propref{a=bc}, \\
(\bigvee,\bigvee): & M \cong P \ot Q,\\
(\bigvee,\bigwedge): & M \cong P \ot (U_1^\epsilon \oplus \ldots \oplus U_q^\epsilon),\\
(\bigwedge,\bigvee): & M \cong (V_1^\epsilon \oplus \ldots \oplus V_p^\epsilon) \ot Q,
\end{array}$$
where $\epsilon \in \{0,1\}$. So in the last three cases, we always have that $M$ is isomorphic to the tensor product and the result follows.
\end{proof}

In fact, all results from this last section equally hold for more general nilpotent groups $G = P_1\ldots P_n$ of order $p_1^{a_1}\ldots p_n^{a_n}$.

\end{document}